\documentclass[12pt, reqno]{amsart}
\usepackage{amsmath, amsthm, amscd, amsfonts, amssymb, graphicx, color}
\usepackage[bookmarksnumbered, colorlinks, plainpages]{hyperref}
\hypersetup{colorlinks=true,linkcolor=red, anchorcolor=green, citecolor=cyan, urlcolor=red, filecolor=magenta, pdftoolbar=true}

\newtheorem{theorem}{Theorem}[section]
\newtheorem{lemma}[theorem]{Lemma}
\newtheorem{prop}[theorem]{Proposition}
\newtheorem{corollary}[theorem]{Corollary}
\theoremstyle{definition}

\theoremstyle{remark}
\newtheorem{remark}[theorem]{Remark}
\numberwithin{equation}{section}

\begin{document}
\setcounter{page}{1}

\title{On reverses of the Golden-Thompson type inequalities}

\author[M.B. Ghaemi, V. Kaleibary and S. Furuichi]{Mohammad Bagher Ghaemi, Venus Kaleibary and Shigeru Furuichi}

\address{ School of Mathematics, Iran University of Science and Technology, Narmak, Tehran 16846-13114, Iran.}
\email{\textcolor[rgb]{0.00,0.00,0.84}{mghaemi@iust.ac.ir}}

\address{ School of Mathematics, Iran University of Science and Technology, Narmak, Tehran 16846-13114, Iran.}
\email{\textcolor[rgb]{0.00,0.00,0.84}{v.kaleibary@gmail.com}}

\address{ Department of Information Science, College of Humanities and Sciences, Nihon University,
3-25-40, Sakurajyousui, Setagaya-ku, Tokyo, 156-8550, Japan.}
\email{\textcolor[rgb]{0.00,0.00,0.84}{furuichi@chs.nihon-u.ac.jp}}





\begin{abstract}
 In this paper we present some reverses of the Golden-Thompson type inequalities: 
 Let $H$ and $K$ be Hermitian matrices such that $ e^s e^H \preceq_{ols} e^K \preceq_{ols} e^t e^H$  for
some scalars $s \leq t$, and $\alpha \in [0 , 1]$. Then for all $p>0$ and $k =1,2,\ldots, n$
\begin{align*} \label{}
  \lambda_k (e^{(1-\alpha)H + \alpha K} )  \leq  (\max \lbrace S(e^{sp}), S(e^{tp})\rbrace)^{\frac{1}{p}} \lambda_k (e^{pH} \sharp_\alpha e^{pK})^{\frac{1}{p}},
\end{align*}
where $A\sharp_\alpha B = A^\frac{1}{2} \big ( A^{-\frac{1}{2}} B^\frac{1}{2} A^{-\frac{1}{2}} \big) ^\alpha A^\frac{1}{2}$ is $\alpha$-geometric mean, $S(t)$ is the so called Specht's ratio and $\preceq_{ols}$ is the so called Olson order. The same inequalities are also provided with other constants. The obtained inequalities improve some known results. 
\end{abstract} \maketitle

\section{introduction}

In what follows, capital letters $A,B,H$and $K$ stand for $n \times n$ matrices or bounded linear operators on an $n$-dimentional complex Hilbert space  $(\mathcal{H}, \langle \;\cdot\; \rangle)$.  For a pair $A,B$ of Hermitian matrices, we say $A \leq B$ if $B-A\geq 0$. Let $A$ and $B$ be two positive definite matrices. For
each $\alpha \in [0, 1]$, the weighted geometric mean $A\sharp_\alpha  B$ of $A$ and $B$ in the sense of
Kubo-Ando \cite{KA} is defined by
\begin{align*}
A\sharp_\alpha B = A^\frac{1}{2} \big ( A^{-\frac{1}{2}} B^\frac{1}{2} A^{-\frac{1}{2}} \big) ^\alpha A^\frac{1}{2}.
\end{align*}
 
\noindent Also for positive definite matrices $A$ and $B$, the weak log-majorization $A \prec_{w log} B$ means that
 \begin{align*}
\prod_{j=1}^k \lambda_j (A) \leq \prod_{j=1}^k \lambda_j (B), \;\;\;\;\;\;\; k=1,2,\cdots ,n,
 \end{align*}
 where $\lambda_1 (A) \geq \lambda_2 (A) \geq  \cdots  \geq \lambda_n (A)$ are the eigenvalues of $A$ listed in decreasing order. If equality holds when $k = n$, we have the log-majorization $A \prec_{log} B$. 
It is known that the weak log-majorization  $A \prec_{w log} B $ implies $\| A \|_u \leq \| B \|_u$ for any unitarily invariant norm $\| \cdot \|_u$ , i.e. $ \| UAV \|_u = \| A \|_u$ for all $A$ and all unitaries $U, V$. See \cite{B1} for theory of
majorization.

In \cite{S2}, Specht obtained an inequality for the arithmetic and geometric means of positive
numbers: Let $x_1 \geq \ldots \geq x_n > 0$ and set $t = x_1/x_n$. Then
 \begin{align*}
 \dfrac{x_1 +\ldots + x_n}{n} \leq S(t) (x_1\ldots  x_n)^{\frac{1}{n}},
 \end{align*}
where
 \begin{align}\label{1}
S(t) =\dfrac{ (t - 1) t^{1/(t-1)}}{e \log t} \;\;\;\;\;(t\neq 1)\;\;\;\; \mbox{and} \;\;\;\; S(1)=1
 \end{align}
is called the Specht ratio at $t$. Note that $\lim_{p\rightarrow 0} S(t^p)^\frac{1}{p} = 1$, $S(t^{-1}) = S(t) >1$ for $t\neq 1, t>0$ \cite{F}. Specht's inequality is a ratio type reverse inequality
of the classical arithmetic-geometric mean inequality. Using this nice ratio we can state our main result in Section 2. 

The Golden-Thompson trace inequality, which is of importance in statistical mechanics and in the theory of random matrices, states that $Tr e^{H+K} \leq Tr e^H e^K$
 for arbitrary Hermitian matrices $H$ and $K$. This inequality has been complemented in several ways \cite{TF, HP}. Ando and Hiai in \cite{TF} proved that for every unitarily invariant norm $\| \cdot \|_u$ and $p > 0$  
\begin{align}\label{17}
\| (e^{pH} \sharp_\alpha e^{pK})^{\frac{1}{p}} \|_u \leq \|e^{(1-\alpha)H + \alpha K} \|_u.
\end{align}

Seo in \cite{S} found some upper bounds on $\|e^{(1-\alpha)H + \alpha K} \|_u$ in terms
of scalar multiples of $\| (e^{pH} \sharp_\alpha e^{pK})^{\frac{1}{p}} \|_u$, which
show reverse of the Golden-Thompson type inequality \eqref{17}.
In this paper we establish another reverses of this inequality, which improve and refine Seo's results. In fact the general sandwich condition $sA \leq B \leq t A$ for positive definite matrices, is the key for our statements. Also, the so called Olson order $\preceq_{ols}$ is used. For positive operators, $A\preceq_{ols} B$ if and only if $A^r \leq B^r$ for every $r \geq 1$\cite{O}. Our results  are parallel to eigenvalue inequalities obtained in \cite{BS} and \cite{Gk}.

 \section{reverse inequalities via specht ratio}
 
To study the Golden-Thompson inequality, Ando-Hiai in \cite{TF} developed the following log-majorizationes:
 \begin{align*}\label{}
A^r \sharp_\alpha B^r \prec_{log} ( A \sharp_\alpha B )^r , \hspace*{1 cm} r\geq 1, 
\end{align*}
or equivalently
\begin{align*}\label{}
 (A^p \sharp_\alpha B ^p)^{\frac{1}{p}} \prec_{log}  ( A^q \sharp_\alpha B^q )^{\frac{1}{q}}, \hspace*{1 cm} 0< q \leq p.
\end{align*}
There are some literatures \cite{S1} on the converse of these inequalities in terms of unitarily invariant norm $\| \cdot \|_u$.
By the following lemmas, we obtain a new reverse of these inequalities in terms of eigenvalue inequalities.

\begin{lemma} \label{l1}
Let $A$ and $B$ be positive definite matrices such that $ s A \leq B \leq t A$ for some scalars $0 < s \leq t$, and $\alpha \in [0 , 1]$. Then
\begin{align} \label{13}
A^r \sharp_\alpha B^r \leq (\max \lbrace S(s), S(t)\rbrace)^r  ( A \sharp_\alpha B )^r, \hspace*{1cm} 0< r \leq 1,
\end{align}
 where $S(t)$ is the Specht's ratio defined as \eqref{1}.
\end{lemma}
\begin{proof}
Let $f$ be an operator monotone function on $[0, \infty)$. Then according to the proof of Theorem 1 in \cite{GK}, we have
\begin{align*}
f (A)\sharp_\alpha f (B)\leq f (M (A\sharp_\alpha B)),
\end{align*}
where $M= \max \lbrace S(s),S(t)\rbrace$. Putting $f(t) = t^r$ for $0< r \leq 1$, we reach inequality \eqref{13}.
\end{proof}

\begin{lemma} \label{l2}
Let $A$ and $B$ be positive definite matrices such that $ s A \preceq_{ols} B \preceq_{ols} t A$ for some scalars $0 < s \leq t$, and $\alpha \in [0 , 1]$. Then
\begin{align} \label{14}
 \lambda_k (A \sharp_\alpha B)^r \leq  \max \lbrace S(s^r), S(t^r)\rbrace   \lambda_k ( A^r \sharp_\alpha B^r ), \hspace*{1cm} r \geq 1,
\end{align}
and hence,
\begin{align} \label{15}
 \lambda_k (A^q \sharp_\alpha B^q)^{\frac{1}{q}} \leq (\max \lbrace S(s^p), S(t^p)\rbrace)^{\frac{1}{p}} \lambda_k ( A^p \sharp_\alpha B^p)^{\frac{1}{p}}, \hspace*{0.5cm}0 < q \leq p,
\end{align}
 where $S(t)$ is the Specht's ratio defined as \eqref{1} and $k=1,2,\ldots, n$.
\end{lemma}
\begin{proof}
First note that the condition $ s A \preceq_{ols} B \preceq_{ols} t A$, is equivalent to the condition $ s^\nu A^\nu \leq B^\nu \leq t^\nu A^\nu$ for every $\nu \geq 1$. In particular, we have $ sA\leq B \leq tA$ for $\nu = 1$. Also, for $r \geq 1$ we have $0 < \frac{1}{r} \leq 1 $ and by \eqref{13}
\begin{align}\label{e1}
A^{\frac{1}{r}} \sharp_\alpha B^{\frac{1}{r}}\leq (\max \lbrace S(s), S(t)\rbrace)^{\frac{1}{r}}  ( A \sharp_\alpha B )^{\frac{1}{r}}.
\end{align}
On the other hand, from the condition $s^\nu A^\nu \leq B^\nu \leq t^\nu A^\nu$ for every $\nu \geq 1$ and letting $\nu =r$, we have  
\begin{align*}
s^r A^r \leq B^r \leq t^r A^r.
\end{align*}
Now if we let $X=A^r$, $Y=B^r$, $w=s^r$ and $z=t^r$,
then
\begin{align}\label{c1}
wX  \leq Y \leq zX.
\end{align}
Using \eqref{e1} under the condition \eqref{c1}, we have
\begin{align*}
X^{\frac{1}{r}} \sharp_\alpha Y^{\frac{1}{r}}\leq (\max \lbrace S(w), S(z)\rbrace)^{\frac{1}{r}}  ( X \sharp_\alpha Y )^{\frac{1}{r}},
\end{align*}
and this is the same as 
\begin{align*}
A \sharp_\alpha B \leq (\max \lbrace S(s^r), S(t^r)\rbrace)^{\frac{1}{r}}  ( A^r \sharp_\alpha B^r )^{\frac{1}{r}}.
\end{align*}
Hence
\begin{align*}
  \lambda_k  (A \sharp_\alpha B) \leq (\max \lbrace S(s^r), S(t^r)\rbrace)^{\frac{1}{r}} \lambda_k ( A^r \sharp_\alpha B^r )^{\frac{1}{r}}.
\end{align*}
By taking $r$-th power on both sides and using the Spectral Mapping Theorem, we get the desired inequality \eqref{14}. Note that from the minimax characterization of eigenvalues of a Hermitian matrix \cite{B1} it follows immediately that $A \leq B$ implies $ \lambda_k  (A) \leq  \lambda_k (B)$ for each $k$. Similarly since $p/q \geq 1$, from inequality \eqref{14}
\begin{align} \label{e2}
 \lambda_k (A \sharp_\alpha B)^{\frac{p}{q}} \leq  \max \lbrace S(s^{\frac{p}{q}} ), S(t^{\frac{p}{q}} )\rbrace   \lambda_k ( A^{\frac{p}{q}}  \sharp_\alpha B^{\frac{p}{q}}  ).
\end{align}
Replacing $A$ and $B$ by $A^q$ and $B^q$ in \eqref{e2}, and using the sandwich condition $ s^q A^q \leq B^q \leq t^qA^q$, we have
\begin{align*} 
 \lambda_k (A^q \sharp_\alpha B^q)^{\frac{p}{q}} \leq  \max \lbrace S(s^{p} ), S(t^{p} )\rbrace  \lambda_k ( A^{p}  \sharp_\alpha B^{p} ).
\end{align*}
This completes the proof.
\end{proof}

Note that eigenvalue inequalities immediately imply log-majorization and unitarily invariant norm inequalities.
\begin{corollary}\label{29} 
Let $A$ and $B$ be positive definite matrices such that $ m I \leq A , B \leq MI$ for some scalars $0 < m \leq M$ with $ h = M/ m$, and let $\alpha \in [0 , 1]$. Then
\begin{align} \label{19}
  A^r \sharp_\alpha B^r \leq S(h)^{r} ( A \sharp_\alpha B )^r, \hspace*{1cm} 0< r \leq 1,
\end{align}
and hence
\begin{align} \label{20}
  \lambda_k (A \sharp_\alpha B)^r \leq S(h^r)  \lambda_k (A^r \sharp_\alpha B^r), \hspace*{1cm} r \geq 1,
\end{align}
\begin{align} \label{21}
    \lambda_k(A^q \sharp_\alpha B^q)^{\frac{1}{q}} \leq S(h^p)^{\frac{1}{p}} \lambda_k( A^p \sharp_\alpha B^p)^{\frac{1}{p}}, \hspace*{1cm} 0 < q \leq p.
\end{align}
 where $S(t)$ is the Specht's ratio defined as \eqref{1} and $k =1,2,\ldots, n$.
\end{corollary}
\begin{proof}
 Since $m I \leq A , B \leq MI$ implies $ \frac{m}{M} A \leq  B \leq \frac{M}{m} A $, the inequality \eqref{19} is obtained by letting $s = m/M$, $ t = M/m$ in Lemma \ref{l1}. Also from $m I \leq A , B \leq MI$, we have $m^\nu I \leq A^\nu , B^\nu \leq M^\nu I$  for every $\nu \geq 1$, and so 
 \begin{align}\label{c2}
 (\frac{m}{M})^\nu  A^\nu \leq B^\nu \leq (\frac{M}{m})^\nu A^\nu.
 \end{align}
 Using Lemma \ref{l2} under the condition \eqref{c2}, we reach inequalities \eqref{20} and \eqref{21}. 
 Note that $S(h) = S(\frac{1}{h})$.
\end{proof}

\begin{remark}
We remark that the matrix inequality \eqref{19} is more stronger than corresponding norm inequality obtained by Seo in \cite[Corollary 3.2]{S}. Also inequality \eqref{21} is presented in \cite[Lemma 3.1]{S}.
\end{remark}

In the sequel we show a reverse of the Golden-Thompson type inequality \eqref{17}, which is our main result.

\begin{theorem} \label{t5}
Let $H$ and $K$ be Hermitian matrices such that $ e^s e^H \preceq_{ols} e^K \preceq_{ols} e^t e^H$ for
some scalars $s \leq t$, and let $\alpha \in [0 , 1]$. Then for all $p>0$,
\begin{align*} \label{}
  \lambda_k (e^{(1-\alpha)H + \alpha K} )  \leq  (\max \lbrace S(e^{sp}), S(e^{tp})\rbrace)^{\frac{1}{p}} \lambda_k (e^{pH} \sharp_\alpha e^{pK})^{\frac{1}{p}},
\end{align*}
where $S(t)$ is the so called Specht's ratio defined as \eqref{1} and $k =1,2,\ldots, n$.
\end{theorem}
\begin{proof}
Replacing $A$ and $B$ by $e^H$ and $e^K$ in the inequality \eqref{15} of Lemma \ref{l2}, we
can write
\begin{align*}
 \lambda_k (e^{qH} \sharp_\alpha e^{qK})^{\frac{1}{q}} \leq  (\max \lbrace S(e^{sp}), S(e^{tp})\rbrace)^{\frac{1}{p}} \lambda_k ( e^{pH} \sharp_\alpha e^{pK})^{\frac{1}{p}}, \hspace*{0.8cm} 0 < q \leq p.
\end{align*}
 By \cite[Lemma 3.3]{HP}, we have
 \begin{align*}
 e^{(1-\alpha)H + \alpha K} = \lim_{q\rightarrow 0}  (e^{qH} \sharp_\alpha e^{qK})^{\frac{1}{q}},
\end{align*}
and hence it follows that for each $p > 0$, 
\begin{align*} \label{}
  \lambda_k (e^{(1-\alpha)H + \alpha K} )  \leq (\max \lbrace S(e^{sp}), S(e^{tp})\rbrace)^{\frac{1}{p}}  \lambda_k (e^{pH} \sharp_\alpha e^{pK})^{\frac{1}{p}}.
\end{align*} 
\end{proof}

\begin{corollary} \label{}
Let $H$ and $K$ be Hermitian matrices such that $ e^s e^H \preceq_{ols} e^K \preceq_{ols} e^t e^H$ for
some scalars $s \leq t$, and let $\alpha \in [0 , 1]$. Then for every unitarily invariant norm $\| \cdot \|_u$ and all $p>0$, 
\begin{align} \label{9}
\| e^{(1-\alpha)H + \alpha K} \|_u  \leq (\max \lbrace S(e^{sp}), S(e^{tp})\rbrace)^{\frac{1}{p}}  \|(e^{pH} \sharp_\alpha e^{pK})^{\frac{1}{p}}\|_u,  
\end{align}
and the right-hand side of \eqref{9} converges to the left-hand side as $p \downarrow 0$. In particular,
\begin{align*} \label{}
\| e^{H + K} \|_u  \leq \max \lbrace S(e^{2s}), S(e^{2t})\rbrace  \|(e^{2H} \sharp e^{2K})\|_u. 
\end{align*}
\end{corollary}

\begin{corollary}\cite[Theorem 3.3-Theorem 3.4]{S} \label{c3}
Let $H$ and $K$ be Hermitian matrices such that $ mI \leq H, K \leq MI$ for
some scalars $m \leq M$, and let $\alpha \in [0 , 1]$. Then for all $p>0$, 
\begin{align*} \label{}
  \lambda_k (e^{(1-\alpha)H + \alpha K} )  \leq S(e^{(M-m)p})^{\frac{1}{p}}  \lambda_k (e^{pH} \sharp_\alpha e^{pK})^{\frac{1}{p}}  ,\hspace*{1cm} k = 1,2,\ldots,n.
\end{align*}
So,  for every unitarily invariant norm $\| \cdot \|_u$ 
\begin{align*} \label{}
  \| e^{(1-\alpha)H + \alpha K} \|_u  \leq  S(e^{(M-m)p})^{\frac{1}{p}} \| (e^{pH} \sharp_\alpha e^{pK})^{\frac{1}{p}}\|_u,
\end{align*}
and the right-hand side of these inequalities converges to the left-hand side as $p \downarrow 0$.
\end{corollary}
\begin{proof}
From $ mI \leq H, K \leq MI$, we have $ e^{\nu m} \leq e^{\nu H}, e^{\nu K}  \leq e^{\nu M}$ for every $\nu \geq 1$ and so we can derive $e^{m-M} e^H \preceq_{ols} e^K \preceq_{ols} e^{M-m} e^H$. Now the assertion is obtained by applying Theorem \ref{t5} and the fact that for every $t>0$, $S(t)=S(\frac{1}{t})$.
 \end{proof}
   
\section{reverse inequalities via kantorovich constant}
A well-known matrix version of the Kantorovich inequality \cite{MP} asserts that if $A$ and $U$ are two matrices such that $0<mI \leq A \leq MI$ and $UU^* = I$, then 
\begin{align}\label{4}
U A^{-1} U^* \leq \frac{(m+M)^2}{4mM}  (UAU^*)^{-1}.
\end{align}
Let $w>0$. The generalized Kantorovich constant $K(w,\alpha)$ is defined by
\begin{align}\label{5}
K(w, \alpha):= \dfrac{w^\alpha - w}{(\alpha-1)(w-1)} \big( \dfrac{\alpha - 1}{\alpha}\dfrac{w^\alpha-1}{w^\alpha - w} \big)^\alpha,
\end{align}
for any real number $\alpha \in \Bbb{R}$ \cite{F}. In fact, $K(\frac{M}{m}, -1) = K(\frac{M}{m}, 2) $ is the constant
occurring in \eqref{4}.

Now as a result of following statement, we have another reverse Golden-Thompson type  inequality which refines  corresponding inequality in \cite{S}. 
 
\begin{prop} \label{6}
\cite[Theorem 3]{Gk}
 Let $H$ and $K$ be Hermitian matrices such that $ e^s e^H \preceq_{ols} e^K \preceq_{ols} e^t e^H$ for some scalars $ s \leq t$, and let $\alpha \in [0 , 1]$. Then 
 \begin{align}\label{8}
\lambda_k (e^{(1-\alpha)H + \alpha K}) \leq K(e^{p(t-s)},\alpha)^{-\frac{1}{p}}  \lambda_k (e^{pH} \sharp_\alpha e^{pK})^{\frac{1}{p}}, \hspace*{1 cm} p>0,
\end{align}
where $K(w, \alpha)$ is the generalized Kantorovich constant defined as \eqref{5}. 
\end{prop}
\begin{theorem} \label{t2}
 Let $H$ and $K$ be Hermitian matrices such that $ m I \leq K , H \leq M I$ for some scalars $m \leq M$ and let $\alpha \in [0 , 1]$. Then for every $p > 0$
 \begin{align*}
\lambda_k (e^{(1-\alpha)H + \alpha K}) \leq K(e^{2p(M-m)},\alpha)^{-\frac{1}{p}}  \lambda_k (e^{pH} \sharp_\alpha e^{pK})^{\frac{1}{p}},\hspace*{1cm} k = 1,2,\ldots,n,
\end{align*}
and the right-hand side of this inequality converges to the left-hand side as $p \downarrow 0$. In particular, 
 \begin{align*}
\lambda_k (e^{H + K})  \leq \dfrac{e^{2M} + e^{2m}}{2e^{M} e^{m} }  \lambda_k ( e^{2H} \sharp e^{2K}),\hspace*{1cm} k = 1,2,\ldots,n.
\end{align*}
\end{theorem}
 \begin{proof}
 Since $ m I \leq K , H \leq M I$ implies $ e^{m-M} e^H \preceq_{ols} e^K \preceq_{ols} e^{M-m} e^H$, desired inequalities are obtained by letting $s = m-M$ and $t = M-m$ in Proposition \ref{6}. For the convergence, we know that  $\dfrac{2w^\frac{1}{4}}{w^\frac{1}{2}+1} \leq K(w,\alpha)\leq 1$, for every $\alpha \in [0, 1]$. So, for every $p>0$ 
\begin{align*}
1 \leq K(w^p,\alpha)^{-\frac{1}{p}}\leq (\dfrac{2w^\frac{p}{4}}{w^\frac{p}{2}+1})^{-\frac{1}{p}} .
\end{align*}
A simple calculation shows that 
\begin{align*}
\lim_{p\rightarrow 0} -\frac{1}{p} Ln (\dfrac{2w^\frac{p}{4}}{w^\frac{p}{2}+1}) = \lim_{p\rightarrow 0} \dfrac{ Ln (w) (w^\frac{P}{2}-1)}{4 w^\frac{p}{2}+1}=0,
\end{align*}
and hence $\lim_{p\rightarrow 0} (\dfrac{2w^\frac{p}{4}}{w^\frac{p}{2}+1})^{-\frac{1}{p}} = 1 $.  Now by using the sandwich condition and letting $w=e^{2(M-m)}$, we have $\lim_{p\rightarrow 0} K(e^{2p(M-m)},\alpha)^{-\frac{1}{p}} = 1$.
\end{proof}

\begin{remark}
Under the assumptions of Theorem \ref{t2}, Seo in \cite[Theorem 4.2]{S} proved that 
\begin{align*}\label{}
  \| e^{(1-\alpha)H + \alpha K}\|_u  \leq K(e^{(M-m)}, p)^{-\frac{\alpha}{p}} K(e^{2p(M-m)},\alpha)^{-\frac{1}{p}} \| ( e^{pH} \sharp_\alpha e^{pK})^\frac{1}{p} \|_u, \hspace*{1cm} 0 < p \leq 1,
\end{align*}
and
\begin{align*}\label{}
  \| e^{(1-\alpha)H + \alpha K}\|_u  \leq K(e^{2p(M-m)},\alpha)^{-\frac{1}{p}} \| ( e^{pH} \sharp_\alpha e^{pK})^{\frac{1}{p}} \|_u, \hspace*{1cm} p\geq 1.
\end{align*}
But the inequality \eqref{8} shows that the sharper constant for all $p>0$ is $K(e^{2p(M-m)},\alpha)^{-\frac{1}{p}}$. Since   for $0 < p \leq 1$, $K(e^{(M-m)}, p)^{-\frac{\alpha}{p}} \geq 1$ and hence
\begin{align*}
 K(e^{2p(M-m)},\alpha)^{-\frac{1}{p}} \leq K(e^{(M-m)}, p)^{-\frac{\alpha}{p}} K(e^{2p(M-m)},\alpha)^{-\frac{1}{p}}.
\end{align*}

\end{remark}

\section{some related results} 
 
 It has been shown \cite{GK} that if $f :  [0,\infty) \longrightarrow [0,\infty)$ is operator monotone function and $0 <mI \leq A \leq B \leq MI \leq I$ with $h=\frac{M}{m}$, then for all $\alpha \in [0 , 1]$
\begin{align} \label{2}
f(A)\sharp_\alpha f(B)  \leq \exp \big( \alpha ( 1- \alpha) (1-\frac{1}{h})^2 \big) f( A \sharp_\alpha B),
\end{align}
This new ratio has been introduced  by Furuichi and Minculete in \cite{FM}, which is different from Specht ratio and Kantorovich constant. By applying \eqref{2} for $f(t)= t^r$, $0< r \leq 1$ we have the following results similar to Lemma \ref{l1} and Lemma \ref{l2}. 

\begin{lemma}\label{l3}
Let $A$ and $B$ be positive definite matrices such that $0 <mI \leq A \leq B \leq MI \leq I$ with $ h = M/ m$, and let $\alpha \in [0 , 1]$. Then
\begin{align*} \label{}
  A^r \sharp_\alpha B^r \leq   \exp \big( r \alpha ( 1- \alpha) (1-\frac{1}{h})^2 \big) ( A \sharp_\alpha B )^r, \hspace*{1cm} 0< r \leq 1,
\end{align*}
\end{lemma} 
\begin{lemma}
Let $A$ and $B$ be positive definite matrices such that $0 <m I \preceq_{ols} A \preceq_{ols} B \preceq_{ols} MI \preceq_{ols} I$ with $ h = M/ m$, and let $\alpha \in [0 , 1]$. Then for all $k =1,2,\ldots, n$, 
\begin{align*} \label{}
\lambda_k (A \sharp_\alpha B)^r \leq \exp \big( \alpha ( 1- \alpha) (1-\frac{1}{h^r})^2 \big) \lambda_k (A^r \sharp_\alpha B^r), \hspace*{1cm} r \geq 1,
\end{align*}
\begin{align} \label{3}
    \lambda_k(A^q \sharp_\alpha B^q)^{\frac{1}{q}} \leq \exp \big(\frac{1}{p} \alpha ( 1- \alpha) (1-\frac{1}{h^p})^2 \big) \lambda_k( A^p \sharp_\alpha B^p)^{\frac{1}{p}}, \hspace*{1cm} 0 < q \leq p.
\end{align}
\end{lemma}
 
\begin{theorem}\label{t3}
Let $H$ and $K$ be Hermitian matrices such that $ e^mI \preceq_{ols} e^H \preceq_{ols} e^K \preceq_{ols} e^M I \preceq_{ols} I$ for
some scalars $m \leq M$, and let $\alpha \in [0 , 1]$. Then for all $p>0$ and $k = 1,2,\ldots,n$
\begin{align*} \label{}
  \lambda_k (e^{(1-\alpha)H + \alpha K} )  \leq \exp \big(\frac{1}{p} \alpha ( 1- \alpha) (1-\frac{1}{e^{p(M-m)}})^2 \big)   \lambda_k (e^{pH} \sharp_\alpha e^{pK})^{\frac{1}{p}}, 
\end{align*}
and so, for every unitarily invariant norm $\| \cdot \|_u$ 
\begin{align*} \label{}
  \| e^{(1-\alpha)H + \alpha K} \|_u  \leq  \exp \big(\frac{1}{p} \alpha ( 1- \alpha) (1-\frac{1}{e^{p(M-m)}})^2 \big) \| (e^{pH} \sharp_\alpha e^{pK})^{\frac{1}{p}}\|_u.
\end{align*}
\end{theorem}
\begin{proof}
The proof is similar to that of Theorem \ref{t5}, by replacing $A$ and $B$ with $e^H$ and $e^K$, and $h=e^{M-m}$ in the inequality \eqref{3}.
\end{proof}
\begin{remark}
Under the different conditions, the different coefficients are not comparable. But it is known that if we have a certain statement under the sandwich condition $0<sA \leq B \leq tA$, then the same statement is also true under the condition $0<mI \leq A , B \leq MI $ and $0<mI \leq A \leq B \leq MI \leq I$. Hence, we can compare the following special cases:
\begin{enumerate}
\item[(1)] Comparison of the constants in Theorem \ref{t3} and in Theorem \ref{t2}: \\
Let $ e^mI \preceq_{ols} e^H \preceq_{ols} e^K \preceq_{ols} e^M I \preceq_{ols} I$. Operator monotony of $\log(t)$ leads to $mI \leq H \leq K \leq MI \leq I$, and so $mI \leq H, K \leq MI $. Now by applying Theorem \ref{t2} we have
 \begin{align*}
\lambda_k (e^{(1-\alpha)H + \alpha K}) \leq K(e^{2p(M-m)},\alpha)^{-\frac{1}{p}}  \lambda_k (e^{pH} \sharp_\alpha e^{pK})^{\frac{1}{p}}, \hspace*{0.8cm} p>0.
\end{align*}
Also, by Theorem \ref{t3}
\begin{align*} \label{}
  \lambda_k (e^{(1-\alpha)H + \alpha K} )  \leq \exp \big(\frac{1}{p} \alpha ( 1- \alpha) (1-\frac{1}{e^{p(M-m)}})^2 \big)   \lambda_k (e^{pH} \sharp_\alpha e^{pK})^{\frac{1}{p}},\hspace*{0.8cm} p>0.
\end{align*}
Letting $h=e^{M-m}\geq1$, the following numerical examples show that there is no ordering between these inequalities.
\begin{itemize}
\item[(i)] Take $\alpha=\frac{1}{2}$, $p=\frac{1}{2}$ and $h=2$, then we have
\begin{align*}
K(h^{2p},\alpha)^{-\frac{1}{p}}- \exp \big(\frac{1}{p} \alpha ( 1- \alpha) (1-\frac{1}{h^p})^2 \big)  \simeq -0.0134963.
\end{align*}

\item[(ii)] Take $\alpha=\frac{1}{2}$, $p=\frac{1}{2}$ and $h=8$, then we have
\begin{align*}
K(h^{2p},\alpha)^{-\frac{1}{p}}- \exp \big(\frac{1}{p} \alpha ( 1- \alpha) (1-\frac{1}{h^p})^2 \big)  \simeq 0.0631159.
\end{align*}
\end{itemize}
\item[(2)] Comparison of the constants in Lemma \ref{l3} and in Lemma \ref{l1}:\\
Let $0<mI \leq A \leq B \leq MI \leq I$. Then the following sandwich condition is obtained 
\begin{align*}
  m \leq \frac{m}{M} \leq 1 \leq A^{-\frac{1}{2}} B A^{-\frac{1}{2}}\leq \frac{M}{m} \leq \frac{1}{m}.
\end{align*}
Now by letting $s=1$ and $t= \frac{M}{m}=h$ in Lemma \ref{l1}, we get
\begin{align}\label{1}
A^r \sharp_\alpha B^r \leq S(h)^r (A\sharp_\alpha B)^r,  \hspace*{1cm} 0 < r \leq 1.
\end{align}
Also, by Lemma \ref{l3}
 \begin{align} \label{2}
  A^r \sharp_\alpha B^r \leq   \exp \big( r \alpha ( 1- \alpha) (1-\frac{1}{h})^2 \big) ( A \sharp_\alpha B )^r, \hspace*{1cm} 0< r \leq 1.
 \end{align}
It is shown in \cite[Remark 2.4]{FM} that there is no ordering between coefficients of \eqref{1} and \eqref{2}. Therefore, we may conclude evaluation of Lemma \ref{l3} and Lemma \ref{l1} are different.

\end{enumerate}
\end{remark}

\bibliographystyle{amsplain}

\begin{thebibliography}{99}


\bibitem{TF}
        T. Ando and F. Hiai, 
        Log-majorization and complementary Golden-Thompson type inequalities, 
        Linear Algebra Appl. 197/198 (1994) 113--131.
        
                
\bibitem{B1}
        R. Bhatia, 
        Matrix Analysis, 
        Grad. Texts in Math., vol. 169, Springer-Verlag, 1997.

               
\bibitem{BS}
        J.-C. Bourin and Y. Seo, 
        Reverse inequality to Golden-Thompson type inequalities: comparison of $e^{A+B}$ and $e^A e^B$, 
        Linear Algebra Appl. {\bf 426} (2007), 312--316.
        
\bibitem{FM}
         S. Furuichi and N. Minculete, 
        Alternative reverse inequalities for Young’s inequality, 
        J. Math. Inequal. 5 (2011) 595--600.
    
\bibitem{Fu}
        T. Furuta, 
        Operator inequalities associated with Holder-McCarthy and Kantorovich inequalities, 
        J. Inequal. Appl. 2 (1998) 137--148.   
       
       
        \bibitem{F}
         T. Furuta, J. Mi\'{c}i\'{c}, J.E. Pe\v{c}ari\'{c} and Y. Seo, 
         Mond-Pe\v{c}ari\'{c} method in operator inequalities, 
         Monographs in Inequalities 1, Element, Zagreb, 2005.  

\bibitem{GK}
            M.B. Ghaemi and V. Kaleibary, 
           Some inequalities involving operator monotone functions and operator means, 
           Math. Inequal. Appl. 19 (2016) 757--764.

        
        \bibitem{Gk}
        M.B. Ghaemi and V. Kaleibary, 
        Eigenvalue inequalities related to the Ando-Hiai inequality, 
        Math. Inequal. Appl. 20 (2017) 217--223.
      
        \bibitem{HP}
          F. Hiai and D. Petz, 
          The Golden-Thompson trace inequality is complemented, 
          Linear Algebra Appl. 181 (1993) 153--185.
        
        \bibitem{KA}
         F. Kubo and T. Ando, 
         Means of positive linear operators, 
         Math. Ann. 246 (1980) 205--224.
         
         
         \bibitem{MP}
          B. Mond and J.E. Pe\v{c}ari\'{c}, 
          A matrix version of the Ky Fan generalization of the Kantorovich inequality, 
         Linear and Multilinear Algebra. 36 (1994) 217--221.
        
        
        \bibitem{O}
         M.P. Olson, 
         The selfadjoint operators of a von Neumann algebra from a conditionally complete lattice,
         Proc. Amer. Math. Soc. 28 (1971) 537--544.
        
      
        \bibitem{S}
         Y. Seo, 
         Reverses of the Golden--Thompson type inequalities due to Ando--Hiai--Petz, 
         Banach J. Math. Anal. 2 (2008) 140--149.
        
        \bibitem{S1}
         Y. Seo, 
         On a reverse of Ando-Hiai inequality, 
         Banach J. Math. Anal. 4 (2010) 87--91.
         
        \bibitem{S2}
          W. Specht, 
          Zur Theorie der elementaren Mittel, 
         Math. Z. 74 (1960) 91--98.


        

\end{thebibliography}

\end{document}